\documentclass[12pt]{article}
\usepackage{amsmath}
\usepackage{amssymb}
\usepackage{amsthm}
\usepackage{amsfonts}
\usepackage{array}
\usepackage[mathscr]{euscript}
\usepackage{graphicx}
\usepackage{epstopdf}
\usepackage{cite}
\begin{document}
\newtheorem{thm}{Theorem}
\newtheorem{Def}{Definition}
\newtheorem{lem}{Lemma}
\newtheorem{cor}{Corollary}
\newtheorem{rem}{Remark}
\newtheorem{note}{Note}
\newtheorem{Result}{Result}
\begin{center}
\Large{\bf On the Secure Vertex Cover Pebbling Number}
\end{center}      
\begin{center}
Glenn H. Hurlbert$^{[1]}$, Lian Mathew$^{[2]}$, Jasintha Quadras$^{[3]}$, S. Sarah Surya$^{[4]}$,  \\
$^{[1]}$Department of Mathematics and Applied Mathematics\\
Virginia Commonwealth University, Virginia, USA\\
$^{[2,4]}$Department of Mathematics\\
Stella Maris College(Autonomous), Chennai, India\\
Affiliated to the University of Madras, India\\
E-mail: lianmathew64@gmail.com\\
$^{[3]}$Department of Mathematics\\
Maris Stella College(Autonomous), Vijayawada, India\\
\end{center}              
\begin{abstract}
A new graph invariant called the secure vertex cover pebbling number, which is a combination of two graph invariants, namely, `secure vertex cover' and `cover pebbling number', is introduced in this paper.  The secure vertex cover pebbling number of a graph $G$, is the minimum number $m$ so that every distribution of $m$ pebbles can  reach some secure vertex cover of $G$ by a sequence of pebbling moves. In this paper, the complexity of the secure vertex cover problem and secure vertex cover pebbling problem are discussed. Also, we obtain some basic results and the secure vertex cover pebbling number for complete $r$- partite graphs, paths, Friendship graphs, and wheel graphs. 
\end{abstract}


\noindent\textbf{2010 Subject Mathematics Classification:} 05C38, 05C70, 05B40

\section{Introduction}
Graph pebbling is one of the rapidly developing areas of research in Graph theory, suggested by Lagarias and Saks and introduced into the literature by F.R.K Chung. Let $G(V, E)$ be a simple connected graph. A distribution $C$ of pebbles on $G$ is defined as a function $C: V \rightarrow \mathscr{N}$, the non-negative integers. A pebbling move is defined as the removal of two pebbles from one vertex and placing one on an adjacent vertex. Then we define the pebbling number of a graph $G$, $f(G)$, is the least $n$ such that, from any distribution of $n$ pebbles on $V(G)$, we can place a pebble to any of the root vertex by a sequence of pebbling moves\cite{CHy}. For a survey of additional results, refer\cite{{HSu},{Hu13},{GrJoJa03},{HuFr00}}. Also, other varations  of graph pebbling can be found in \cite{ {GDom},{BuChCr08},{ChHaHe17},{BeBrCz03},{Kn19},{GaMa21},{SaLiMa},{CCov}, {Cov}, {LoDhKi22}, {SLSTec}}.

\subsection{Definitions}
\begin{Def}
\noindent A \textbf{vertex cover} of a graph $G$ is a subset $K$ of $V$ such that every edge of $G$ has at least one end in $K$\cite{Sec}.
\end{Def}
\begin{Def}
A \textbf{secure vertex cover} is a set $S\subseteq V,$ such that $S$ is a vertex cover and for each $u\in V$, there exists a $v\in S\cap N(u)$ such that $(S\backslash \{v\})\cup\{u\}$ is a vertex cover. The minimum number of vertices in the set $S$ is called the \textbf{secure vertex cover number} denoted by $\alpha_{s}(G)$\cite{Sec}.
\end{Def}
\noindent We say that a distribution $C$ can reach a set of vertices $S$ if one can make a sequence of pebbling moves from $C$ that result in placing at least one pebble on each vertex of $S$.
\begin{Def}
\noindent The \textbf{cover pebbling number} denoted by $\gamma(G)$ of a graph $G$ is the minimum number $m$ so that every distribution of $m$ pebbles can reach $V(G)$\cite{CCov}.
\end{Def}
\begin{Def}
\noindent The \textbf{covering cover pebbling number} denoted by $\sigma(G)$ is the minimum number $m$ so that every distribution of $m$ pebbles can reach some vertex cover of $G$\cite{Cov}. 
\end{Def}
\noindent In this paper, we introduce the following new graph invariant. 
\begin{Def}
The \textbf{secure vertex cover pebbling number} of a graph $G$, denoted by  $f_{svcp}(G)$, is the minimum number $m$ so that every distribution of $m$ pebbles can  reach some secure vertex cover of $G$.
\end{Def}

\noindent Note that a secure vertex cover for a given graph $G$ may not be unique which makes this problem quite challenging.\\
 
\noindent For example, Figure 1 displays eight different secure vertex covers for the path $P_6$. These are all the secure vertex covers for $P_6$ up to symmetry and the possibilities  are 1245, 1246, 1345, 1346, 1356,  2345, 2346, 2356. Here, finding the minimum number of pebbles required to place one pebble on the secure vertex cover set under any distribution of pebbles to the vertices of path $P_6$ makes this problem a difficult one.
\begin{figure}[h!]
\centering
\includegraphics[width=7 cm,height=7cm]{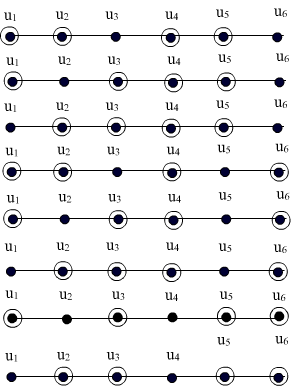}
\caption{Secure vertex cover sets of $P_6$}
\label{Sel}
\end{figure}
\begin{Def}
Let $G_1(V_1,E_1)$ and $G_2(V_2,E_2)$ be a connected simple graph. Then
$G_1\cup G_2$ is the graph $G(V, E)$ where $V= V_1\cup V_2$ and $E=E_1\cup E_2$ and $G_1 + G_2$ is $G_1\cup G_2$ together with the edges joining elements of $V_1$ to elements of $V_2$\cite{BMG}.
\end{Def}

\noindent 
The following provides the impetus for us to introduce this topic: Let the villages(hamlets) of a city be denoted by the vertices and let two vertices be connected by an edge if there is a road that runs between them. The bare minimum of gas stations needed in that city can then be found by finding the secure vertex cover of the obtained graph. As the fuel is wasted during transportation as heat or energy, by determining the secure vertex cover pebbling number, we can determine the least amount of gasoline needed for the city.\\

\noindent In this paper,  some basic results, the complexity and the secure vertex cover pebbling number for certain families of graphs such as complete $r$- partite graph, path $P_n$, friendship graph $F_{n}$ and wheel graph $W_n$ are determined.

\section{Results}
The following three results are straightforward.
\begin{thm}
 For a simple connected graph $G$ with diameter $d$, $n-1\leq f_{svcp}(G)\leq (n-1)2^{d-1}$ and the equality holds for $K_2$.
\end{thm}
\begin{thm}
For a simple connected graph $G$, $\sigma(G)\leq f_{svcp}(G)\leq \gamma(G)$.
\end{thm}
\begin{thm}
For a complete graph $K_n$ on $n$ vertices, the secure vertex cover pebbling number, $f_{svcp}(K_n)= 2n-3$.
\end{thm}

\noindent The next two theorems require only short proofs. 
\begin{thm}
The secure vertex cover pebbling number of join of two graphs $G(V,E)$ and $G^{'}(V^{'}, E^{'})$ is $f_{svcp}(G+G^{'})\leq \gamma (G+G^{'})-4$, and the equality holds when $G$ and $G^{'}$ are complete graphs.
\end{thm}
\begin{proof}
Clearly, any $|V|+|V^{'}|-1$ vertices form a secure vertex cover for the join of two graphs $G$ and $G^{'}$. Also, the maximum number of pebbles needed to place a pebble on a vertex of $G+G^{'}$ is four and hence the result follows.
\end{proof}

\begin{thm}
For complete graphs $K_m$ and $K_n$, $m,n\geq 1$, $f_{svcp}(K_m+K_n)= f_{svcp}(K_m)+f_{svcp}(K_n)+3.$
\end{thm}
\begin{proof}
By the definition of the join of two graphs, we know that $K_m+K_n$ is also a complete graph with $m+n$ vertices. Therefore, by theorem 3 we have,
\begin{align*}
 f_{svcp}(K_m+K_n)&=2(m+n) - 3\\
&=(2m-3)+(2n-3)+3\\
&=f_{svcp}(K_m)+f_{svcp}(K_n)+3
\end{align*}
\end{proof}

\noindent In Section 2.1, we prove that the secure vertex cover problem and the secure vertex cover pebbling problem are both NP- complete. In section 2.2, we prove exact results for the secure vertex cover pebbling number of complete $r$- partite graphs, paths, wheels, and friendship graphs. 

\subsection{Complexity}
\subsubsection{Complexity of Secure Vertex Cover Problem}
In this section, we prove that the secure vertex cover problem is NP-complete. The proof's reduction is from the vertex cover problem which is a known NP-complete problem\cite{GaJo90}. The vertex cover problem asks `for a given graph $G$ and an integer $k$, does the graph $G$ contain a vertex cover of cardinality at most $k$?'. 
\begin{thm}
The secure vertex cover problem is NP- complete.
\end{thm}
\begin{proof}
Let $G(V,E)$ be a given graph. Then, construct a graph $G^{*}(V^{*}, E^{*})$ as follows: Let the vertex set of $V^{*}$ be given by, $V^{*}(G^{*}) = V \cup V^{'}$ where $V^{'}= \{v^{'}: v\in V\}$ and the vertex set $V^{'}$ induces a clique. Let the edge set of $E^{*}$ be given by, $E^{*}(G^{*})= E \cup E^{'}\cup E^{"}$ where $E^{'}$ contains the edges of the complete graph induced by the vertices of $V^{'}$ and $E^{"}$ contains the edges which join the vertices of $V$ to the corresponding vertices of the vertex set $V^{'}$. \\
Thus, the obtained graph $G^{*}(V^{*}, E^{*})$ has two layers, the top layer with graph $G$ and the bottom layer with $|V|$ vertices which forms a clique(see Figure 2). Also, note that the construction of $G^{*}$ from $G$ can be done in polynomial time. \\
Let $S$ be the vertex cover for the graph $G$ with $|S|\leq k$. To complete the proof, we must show that $S$ is a vertex cover for the graph $G$ if and only if $S\cup V^{'}$ is a secure vertex cover for the graph $G^{*}$.\\
Let us assume that $S$ with $|S|\leq k$ is a vertex cover for graph $G$. To prove $D=S\cup V^{'}$ forms a secure vertex cover for $G^{*}$. It is obvious that $|V^{'}|$ vertices form a secure vertex cover for $V^{'}$. Also, for every vertex $v\in V\backslash S$, there exists a vertex $v^{'}\in V^{'}\in D$, and $v\in N(v^{'})$ such that $D\backslash \{v^{'}\} \cup \{v\}$ forms a vertex cover for $G^{*}$. Therefore, $D$ is a secure vertex cover for $G^{*}$ whenever $S$ is a vertex cover for $G$.\\
Conversely, suppose that $D=S\cup V^{'}$ forms a secure vertex cover for $G^{*}$. From the definition, $S\cup V^{'}$ is also a vertex cover for $G^{*}$. So, even if we remove $V^{'}$ from $D$, $S$ forms a vertex cover for $G$. Therefore, $S$ is a vertex cover of $G$ whenever $S\cup V^{'}$ is a secure vertex cover of $G^{*}$.
\end{proof}
\begin{figure}[h!]
\centering
\includegraphics[width=7cm,height=5cm]{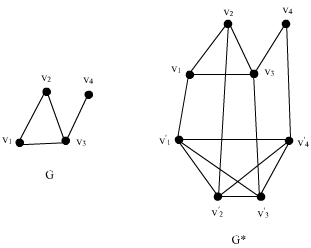}
\caption{Construction of $G^{*}$ from $G$}
\label{Sel}
\end{figure}
\subsubsection{Complexity of the Secure Vertex Cover Pebbling Problem}
In this section, we prove that the secure vertex cover pebbling problem is NP- complete. The proof's reduction is from the secure vertex cover problem which is a NP-complete problem(Theorem 7). The secure vertex cover problem asks `for a given graph $G$ and an integer $k$, does the graph $G$ contain a secure vertex cover of cardinality at most $k$?'. 

\begin{thm}
The secure vertex cover pebbling problem is NP- complete.
\end{thm}
\begin{proof}
Let $G(V,E)$ be a given graph. Then, construct a graph $G^{*}(V^{*}, E^{*})$ as follows: Let the vertex set of $V^{*}$ be given by, $V^{*}(G^{*})= V \cup \{u\}$ and the edge set of $E^{*}$ be given by $E^{*}(G^{*})= E\cup E^{'}$ where $E^{'}$ contains the edges which joins the vertices of $V$ to the vertex $u$.\\
So, the obtained graph $G^{*}(V^{*}, E^{*})$ has two layers, the top layer with graph $G$ and the bottom layer with vertex $u$ (see Figure 3). Since, we have added only one extra vertex, the construction of $G^{*}$ from $G$ can be done in polynomial time.\\
Let $S$ be the secure vertex cover for the graph $G$ with $|S|\leq k$. To complete the proof, we will show that $S$ is a vertex cover for the graph $G$ if and only if the secure vertex cover pebbling number is $l+4(k-1)+2$, for some $1\leq l\leq 3$.\\
Let us assume that $S$ is a secure vertex cover for the graph $G$ with $|S|\leq k$. To prove that the minimum number of pebbles required to produce a secure vertex cover solution is $l+4(k-1)+2,$ for some $1\leq l\leq 3$, note that to produce a secure vertex cover solution, we need to place a pebble on all the vertices of set $S$ and the vertex $u$.\\
Consider the distribution of placing all the pebbles on any vertex $v$ of $G$. Then, we require a minimum of $1+4(k-1)+2$ or $2+4(k-1)+2$ or $3+4(k-2)+2$ pebbles in order to produce a secure vertex cover solution if the vertex $v$ belongs to the set $S$, if the vertex $v$ does not belongs to the set $S$ but an adjacent vertex of $v$ belongs to the set $S$, if the vertex $v$ and the adjacent vertex of $v$ belongs to the set $S$, respectively. Consequently, a minimum of $l+4(k-1)+2, 1\leq l\leq 3$ pebbles are required  to produce a secure vertex cover solution. \\
Now, we will prove the upper bound by considering all the possibilities. First, consider the distribution of all the pebbles to the vertices of $G$. In order to produce a secure vertex cover solution, we are forced to place a pebble on the vertex $u$. If not, we are forced to place a pebble on all the vertices of $G$ which is a contradiction and for placing a pebble on the vertex $u$ by a pebbling move we require a maximum of two pebbles since all the vertices of $G$ are adjacent to the vertex $u$. Note that the maximum distance from any vertex of a graph $G$ to any other vertex of graph $G$ is two and hence we require a maximum of $4(k-1)$ pebbles in order to place a pebble on $k-1$ target vertices on the set $S$ by a sequence of pebbling moves. Now, if the left-out target vertex has pebbles on it, or if one of the adjacent vertex of the target vertex has two pebbles on it, then there is nothing to prove. Now, consider the case where one of the vertex say, $v$ in $G$ has three pebbles on it and the vertex $v$ as well as one of the adjacent vertex of $v$ belongs to the target set $S$. Then, a secure vertex cover solution can be obtained as follows: retain one pebble on the vertex $v$ and place a pebble on the adjacent vertex of $v$ by a pebbling move. Consequently, we have placed pebbles on two target vertices of the graph $G$ and in order to place a pebble on the $k-2$ target vertices we require only a maximum of $4(k-2)$ pebbles and we are done. \\
Now, consider the distribution of all the pebbles to the vertex $u$. Then, a maximum of $2k$ pebbles are required to place a pebble on the $k$ vertices of set $S$ by a sequence of pebbling moves and by retaining one pebble on the vertex $u$, we are done. \\
Now, consider the distribution of pebbles to the vertices of $G^{*}$. Then, it is straight forward to see that we require less number of pebbles compared to the former case.
Thus, $l+4(k-1)+2, 1\leq l\leq 3$ pebbles are only required to produce a secure vertex cover solution whenever $S$ is a secure vertex cover for $G$.\\
Conversely, assume that $l+4(k-1)+2, 1\leq l\leq 3$ pebbles are required to produce a secure vertex cover solution. Then, to prove $S$ is a secure vertex cover for $G$. Now, we are forced to choose the vertex $u$ in the secure vertex cover set for $G^{*}$. If not, then we are forced to choose all the vertices on the graph $G$ in the secure vertex cover set since all the vertices of graph $G$ are incident with vertex $u$ which is a contradiction. Thus, vertex $u$ belongs to the secure vertex cover set of $G^{*}$ and we require a maximum of two pebbles in order to place a pebble on the vertex $u$ by a pebbling move. Now, from the definition, it is straight forward to see that after a sequence of pebbling moves, the remaining pebbles are placed on the vertices of the secure vertex cover of $G$. Thus, $S$ forms a vertex cover for $G$. Therefore, $S$ is a secure vertex cover for $G$ whenever $l+4(k-1)+2, 1\leq l\leq 3$ pebbles are required to produce a secure vertex cover solution.
\end{proof}
\begin{figure}[h!]
\centering
\includegraphics[width=10 cm,height=4cm]{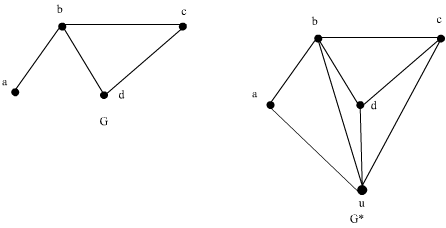}
\caption{Construction of $G^{*}$ from $G$}
\label{Sel}
\end{figure}
\subsection{Exact Results}
\subsubsection{Secure vertex cover pebbling number for complete multi-partite graphs}
We prove the following lemma in order to find the secure vertex cover pebbling number for complete $r$- partite graph $K_{p_1,p_2,p_3,...,p_r}$.
\begin{lem}
For $p_1\geq p_2\geq ...\geq p_r$, let $G= K_{p_1,p_2,p_3,...,p_r}$ be the complete $r$- partite graph with $p_1,p_2,p_3,...,p_r$ vertices in the vertex classes $s_1,s_2,s_3,...,s_r$ respectively. Then, the secure vertex cover number is given by, $\alpha_{s}(G)=p_1+p_2+...+p_r-1$.
\end{lem}
\begin{proof}
Let $n=p_1+p_2+...+p_r$. Then, it is obvious that any $n-1$ vertices of graph $G$ forms a secure vertex cover which gives the upper bound.\\
To prove the lower bound, let us assume that some $n -2$ vertices form the secure vertex cover $S^{'}$ for $G$. Let us assume that the two left-out vertices are from the same partition $s_1$, say, $t_{11}$ and $t_{12}$. Then, $(S^{'}\backslash \{t_{21}\})\cup\{t_{11}\}$ is not a vertex cover which is a contradiction. Now, let us assume that the two left out vertices are from different partitions $s_1$ and $s_2$ say, $t_{11}$ and $t_{21}$. Then, the set $S^{'}$ is not a vertex cover and hence not a secure vertex cover which is again a contradiction. \\
Therefore, $\alpha_{s}(G)=n-1$.
\end{proof}
\begin{Def}
The support, denoted $supp(C)$, of a distribution $C$ is the set of vertices containing pebbles.
\end{Def}
\begin{thm}
For $p_1\geq p_2\geq ...\geq p_r$, $p_1\geq 2$, the secure vertex cover pebbling number of complete $r$- partite graph $G= K_{p_1,p_2,p_3,...,p_r}$ is, $f_{svcp}(G)=4p_1+2p_2+...+2p_r-7$.
\begin{proof}
Consider the distribution of placing all the pebbles on any one of the vertices of the graph $G$. Then, to place a pebble on any $n-1$ vertices of the graph $G$, we require a minimum of $4p_1+2p_2+...+2p_r-7$ pebbles. Hence, $f_{svcp}(G)\geq 4p_1+2p_2+...+2p_r-7$. \\
For the upper bound, let $C$ be a distribution of pebbles on $G$ and denote $C_i = C_{V_i}$ (the distribution $C$ restricted to the part $V_i$).\\
Suppose there is a part $V_i$ with $C(V_i)>1$ and $|supp(C_i)| > 1$.  Label the vertices of $V_i$ as $x_1, ..., x_{p_i}$ so that $C(v_1) \geq C(x_j)$ for all $j$, and choose some $k$ such that $C(x_k) > 0$.  If $C(x_k) = 1$ then define the distribution $C^{'}$ by $C^{'}(x_1) = C^{'}(x_1) + 1, C^{'}(x_k) = 0$, and $C^{'}(x) = C(x)$ for all other $x$.  If $C(x_k) > 1$ then define the distribution $C^{'}$ by $C^{'}(x_1) = C^{'}(x_1) + 2, C^{'}(x_k) = C(x_k) - 2$, and $C^{'}(x) = C(x)$ for all other $x$.  In either case it is easy to see that  $|C^{'}| \geq |C|$ and  for any set $S$ of $n-1$ vertices, if $C^{'}$ reaches $S$ then so does $C$.  Thus, we may restrict our attention to the case in which $|supp(C_i)| \leq 1$ for all $i$.\\
Suppose there are $i < j$ with pebbles in each of $V_i$ and $V_j$.  Label those vertices $u_i$ in $V_i$ and $u_j$ in $V_j$.  If $C(u_j) = 1$, then define the distribution $C^{'}$ by $C^{'}(u_i) = C^{'}(u_i) + 1, C^{'}(u_j) = 0$, and $C^{'}(u) = C(u)$ for all other $u$.  If $C(u_j) > 1$, then define the distribution $C^{'}$ by $C^{'}(u_i) = C^{'}(u_i) + 2, C^{'}(u_j) = C(u_j) - 2$, and $C^{'}(u) = C(u)$ for all other $u$.  In either case it is easy to see that $|C^{'}| \geq |C|$ and for any set $S$ of $n-1$ vertices, if $C^{'}$ reaches $S$ then so does $C$.  Thus, we may restrict our attention to the case in which only one $V_i$ contains pebbles; that is, $C$ is stacked on only one vertex $y$ of $G$.\\
Suppose that $y$ is not in $V_1$, and let $z$ be any vertex of $V_1$.  Define $C^{'}$ by $C^{'}(z) = C(y)$ and $C^{'}(v) = 0$ for all other $v$.  Then, it is easy to see that $|C^{'}| = |C|$ and for any set $S$ of $n-1$ vertices, if $C^{'}$ reaches $S$ then so does $C$.  Thus, we may restrict our attention to the case in which $C$ is stacked on one vertex $z$ of $V_1$.\\
Now, if $|C| \geq 4p_1 + 2p_2 + ... + 2p_r - 7$, then it is clear to see that $C$ reaches $V-v$, where $v$ is any vertex of $V_1$ different from $z$, which exists because $p_1\geq 2$.
\end{proof}
\end{thm}
\begin{cor}
For a star graph $K_{1,n}$, $f_{svcp}(K_{1,n})= \sigma(K_{1,n})$. 
\end{cor}
\subsubsection{Secure vertex cover pebbling number for path $P_n$}
\begin{Def}
Let $P_n= v_1,v_2,...,v_n$ denote the path on $n$ vertices. Define the subset $S_n\subset V(P_n)$ as follows:
\begin{itemize}
\item for $n\geq 2, v_{n-1}\in S_n$;
\item for $n\geq 3, v_{n-2}\in S_n$;
\item for $n\geq 5, v_{n-4}\in S_n$; and
\item for $n\geq 6$, we have:
  \begin{itemize}
            \item for all $2\leq k \leq n-5, v_k \in S_n$ if and only if $v_{k+5}\in S_n;$ and
            \item $v_1 \in S_n$ if and only if $n \not\equiv 4 \pmod{5}$
  \end{itemize}
\end{itemize}
\end{Def}
For example,
\begin{align*}
S_4&= \{v_2,v_3\},\\
S_{11}&=\{v_1,v_2,v_4,v_5,v_7,v_9,v_{10}\}, \text{and}\\
S_{18}&= \{v_1,v_2,v_4,v_6,v_7,v_9,v_{11},v_{12},v_{14},v_{16},v_{17}\}
\end{align*}
Because $\{v_i, v_{i+1}\}\cap S_n \neq \phi$ for all $1\leq i\leq n$, $S_n$ is a vertex cover of $P_n$.
\begin{Def}
For any subset $S\subseteq V(P)$, define its weight $w(S)=\sum_{v_j\in S} 2^{j-1}.$
\end{Def}

\noindent Note that $w(S)$ equals the number of pebbles a distribution stacked on $v_1$ needs to have in order to reach $S$ via pebbling steps. Further note that the secure vertex cover in Figure 1 are ordered from least to greatest weight.
\begin{Def}
Let $Sec(n)$ be the set of secure vertex covers of $P_n$ and define $w_n=min_{S\in Sec(n)}w(S)$.
\end{Def}
\noindent It is not difficult to see that $S_n\in Sec(n)$. Indeed, suppose that $v_i\notin S_n$. If $v_{i-2}\in S_n$ then $(S_n-\{v_{i-1}\})\cup \{v_i\}$ is a vertex cover. If $v_{i+2}\in S_n$ then $(S_n - \{v_{i+1}\}\cup \{v_i\}$ is a vertex cover. The only instance not handled by these two cases is when $i=2$, in which case $(S_n-\{v_1\})\cup \{v_2\}$ is a vertex cover.

\begin{Def}
For any $S\in Sec(n)$, define $S^{'}$ by $v_i\in S^{'}$ if and only if $v_{i+1}\in S$, for all $ 1\leq i\leq n-1.$
\end{Def}
\noindent We see that if $S\in Sec(n)$ then $S^{'} \in Sec(n-1)$. Moreover, $S=S_n$ implies that $S^{'}= S_{n-1}$. Additionally, $w(S^{'})=(w(S)-\chi_{S}(v_1))/2$, where the indicator function $\chi_X(v_i)=1$ if $v_i\in X$ and 0 otherwise.
\begin{lem}
For every integer $n\geq 2$ and every $S\in Sec(n)$, $w(S)=w_n$ if and only if $S=S_n$.
\end{lem}
\begin{proof}
We proceed by induction and notice that the statement is true for $n=2$. Now we suppose that $n=k\geq 3$, and that the statement is true for every value of $2\leq n< k$.\\
Now suppose there is some $T\in Sec(k)$ such that $w(T)< w(S_k)$. Then
\begin{align*}
w(T^{'})&=(w(T)-\chi_T(v_1))/2\\
&< (w(S_k)-\chi_T(v_1))/2\\
&=(w(S_k)-\chi_{S_k}(v_1))/2+(\chi_{S_k}(v_1)-\chi_T(v_1))/2\\
&=w(S_{k-1})+(\chi_{S_k}(v_1)-\chi_T(v_1))/2\\
&\leq w(S_{k-1})+1/2
\end{align*}
and so $w(T^{'})\leq w(S_{k-1})$. However, the induction hypothesis implies that $w(T^{'})\geq w(S_{k-1})$, which implies that $w(T^{'})=w(S_{k-1})$, and consequently that $T^{'}= S_{k-1}$. Therefore, having $w(T)\leq w(S_k)$ means that $v_1\in S_k$(so $k\not\equiv  4\pmod{5})$ and $T= S_k - \{v_1\}$. But if $k$ mod 5 $\in \{0,2\}$ then $T$ is not a vertex cover, while if $k$ mod 5 $\in \{1,3\}$ then $T$ is not a secure vertex cover. Hence no such $T$ exists. 
\end{proof}
\begin{cor}
For $n\geq 0$, let $s=\lfloor\frac{n}{5}\rfloor$, $l=n$ mod 5 and $q_s=(2^{5s}-1)/31$. Then, $w_0=0, w_1=w_2=1, w_3=3, w_4=6$, and $w_n=2^l(13q_s)+w_l$ for all $n\geq 5$.
\end{cor}
\noindent For any subset $S\subseteq V(P_n)$, define $S^{-}$ by $v_i \in S^{-}$ if and only if $v_{n+1-i}\in S$. For any $i\leq j$, denote the subpath of $P_n$ by $P[i,j]=\{v_i,...,v_j\}$ with $S[i,j]=S\cap P[i,j]$. Also, define $q_s=Q_s-Q_{s-1}=2^{5(s-1)}$, and note that $13 q_s> w_{n-s}.$
\begin{thm}
For all $n\geq 0$, we have $f_{svcp}(P_n)=w_n$.
\end{thm}
\begin{proof}
Let $v_1, v_2, ..., v_{n}$ be the vertices of the path $P_{n}, n\geq 5$ (the theorem is true for $n\leq 4$). In addition, Lemma 2 implies that $f_{svcp}(P_n)\geq w_n$ for all $n$, so we only need to prove the upper bound.\\

\noindent The upper bound can be proved by the method of induction on $n$. Now, consider the distribution of all the $w_n$ pebbles on the vertices of the path $P_{n}$. \\
We will show that this many pebbles is sufficient to reach either $S_n$ or $S_n^{-}$.
\begin{itemize}
\item If $|C(P[1,n-5])|\geq w_{n-5}$, then we are done: those pebbles reach $S_n[1,n-5]=S_{n-5}$ by induction, and the remaining $w_n-w_{n-5}=13q_s$ pebbles reach $S_n[n-4,n]$.
\item Otherwise 
\begin{align*}
|C(P[6,n])|&\geq |C(P[n-4,n])|\\
&=|C|-|C(P[1,n-5])|\\
&> w_n - w_{n-5}\\
&=13q_s\\
&>w_{n-5}
\end{align*}
and so we are done: those pebbles reach $S^{-}_n[6,n]=S^{-}_{n-5}$ by induction, and the remaining $w_n -w_{n-5}=13 q_s$ pebbles reach $S_n[1,5]$.
\end{itemize}
\end{proof}
\begin{figure}[h!]
\centering
\includegraphics[width=10 cm,height=4cm]{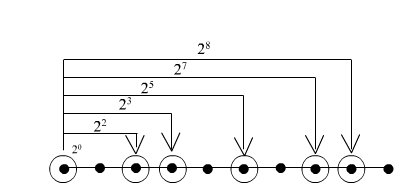}
\caption{Distribution of all the pebbles on the initial vertices of $P_{10}$}
\label{Sel}
\end{figure}
\begin{cor}
For a path $P_n$, $\sigma(P_n)=f_{svcp}(P_n)$ if and only if $n=1,2$.
\end{cor}
\subsubsection{Secure vertex cover pebbling number for some wheel related 
graphs}
\begin{Def}
\noindent A \textbf{Friendship graph} denoted by $F_n$ can be constructed by joining $n$ cycles $C_3$ with a common centre vertex called the apex vertex\cite{Frien}. 
\end{Def}

\noindent We prove the following lemma to find the secure vertex cover pebbling number for friendship graph $F_n$.
\begin{lem}
The secure vertex cover number of the friendship graph $F_n$, $\alpha_{s}(F_n)=n+1.$
\end{lem}
\begin{proof}
Let $u$ be the apex vertex of the friendship graph $F_n$ and let the remaining vertices of the cycle $C_i, 1\leq i\leq n$, be $u_{ij}, 1\leq i\leq n, j=1,2.$ \\
We claim that any one of the $u_{ij}, 1\leq i \leq n, j=1,2$ from each cycle $C_i, 1\leq i \leq n$ and the apex vertex forms a secure vertex cover $S$ for the friendship graph $F_n$. Clearly, the apex vertex $u$ belongs to the set $S$. If not, to cover all the edges, we need all the $2n$ vertices of the cycle in the set $S$ which is a contradiction. Suppose any one of the  $u_{ij}, 1\leq i \leq n-1, j=1,2$ from each cycle $C_i, 1\leq i \leq n-1$ and the apex vertex belongs to the set $S$, then $S$ does not form a vertex cover and hence not a secure vertex cover also. Therefore, $\alpha_{s}(F_n)=n+1.$\\
\begin{thm}
For a friendship graph $F_n$, the secure vertex cover pebbling number, $f_{svcp}(F_n)= 4(n-1)+3$.
\end {thm}
\noindent By Lemma 2, we observe that in order to produce a secure cover solution we need to place at least one pebble on any one vertex of each cycle and the apex vertex. Let the proof be divided into the following cases based on the pebbles distributed to the apex vertex. \\
\textbf{Case 1.} Apex vertex $u$ has at least $2n+1$ pebbles\\
Since, each $u_{ij}$, $1\leq i\leq n$, $j=1,2$ is at a distance of one from the apex vertex, we require a maximum of $2n$ pebbles to place a pebble on $n$ vertices of $n$ cycles. Then, by placing one pebble on the apex vertex, we are done.\\

\noindent \textbf{Case 2.} Apex vertex $u$ has exactly $t< 2n+1$ pebbles\\
In this case, a minimum of $4(n-1)+3-(2n+1)= 2n+2$ pebbles are distributed to the vertices $u_{ij}, 1\leq i\leq n, j=1,2$ of $C_{i}, 1\leq i\leq n$ of $F_n$. Now, retain a maximum of two pebbles on each cycle $C_{i}, 1\leq i\leq n$ and transfer the remaining pebbles to the apex vertex by a sequence of pebbling moves. Let us assume that vertices of $m$ cycles $m\leq n$ of $F_n$ receive a maximum of two pebbles each. \\
\textbf{Case 2.1.} $m=1$ and $t=0$.\\
Here, we have distributed all the $4(n-1)+3$ pebbles to any one cycle $C_{i}, 1\leq i\leq n$, say, $C_1$ of $F_n$. Since, $4(n-1)+3$ is odd, for any distribution of pebbles to $C_1$, we can transfer $\frac{4(n-1)+2}{2}$ pebbles to the hub vertex and it is sufficient to place a pebble on $(n-1)$ vertices of cycle $C_i, 2\leq i\leq n$ and the apex vertex by a sequence of pebbling moves.\\
\textbf{Case 2.2.} $m=1$ and $t\geq 1$.\\
In this case, all the remaining $4(n-1)+3-t$ pebbles are distributed to any one of the cycle $C_{i}, 1\leq i\leq n$ of $F_n$. Retain a maximum of two pebbles on the cycle and transfer the remaining to the apex vertex.\\
Thus, the apex vertex has $t+\lfloor\frac{4(n-1)+3-t-2}{2}\rfloor= t+2(n-1)+\lfloor\frac{1-t}{2}\rfloor$ pebbles. Moreover, in order to produce a secure vertex cover solution we require $2(n-1)+1$ pebbles as we need to place a pebble on any one vertex of the remaining $(n-1)$ cycles and the hub vertex of $F_n$.\\
Since, $ t+2(n-1)+\lfloor\frac{1-t}{2}\rceil- (2(n-1)+1)\geq 0$ as $t\geq 1$, which completes the proof.\\
\textbf{Case 2.3.} $m\geq 2$\\
In this case, we use a maximum of $2m$ pebbles to place a minimum of one pebble on $m$ cycles of $F_n$. Thus, the apex vertex has a minimum of $t+\lfloor\frac{4(n-1)+3-t-2m}{2}\rfloor= t+2(n-1)+1-m+\lfloor\frac{-t}{2}\rfloor$ pebbles on it.\\
However, in order to produce a secure vertex cover solution, we require a minimum of $2(n-m)+1$ pebbles since we need to place at least one pebble on any one vertex of the remaining $n-m$ cycles and the apex vertex $u$ of $F_{n}$.\\
Now, $t+2(n-1)+1-m+\lfloor\frac{-t}{2}\rfloor- (2(n-m)+1)= m-2+t+\lfloor\frac{-t}{2}\rfloor \geq 0 $ as $t\geq 0$ and $m\geq2$, we are done with the proof.
\end{proof}
\begin{cor}
For a friendship graph $F_n$, $\alpha(F_n)=f_{svcp}(F_n)$.
\end{cor}
\begin{Def}
\noindent A \textbf{Wheel graph} denoted by $W_n$ on $n+1$ vertices is the graph obtained from $K_1+C_n$, where $C_n$ is a cycle with $n$ vertices\cite{Cov}.
\end{Def}
\begin{thm}
For a wheel graph, $W_{5s+l}$, $0\leq l\leq 4$, $(5s+l\geq 3)$, the secure vertex cover pebbling number is given by, \\
$f_{svcp}(W_{5s+l})=$
$\begin{cases}
12s-1,\,\,\, l=0\\
12s+3,\,\,\, l=1\\
12s+7, \,\,\,\,l=2,3\\
12s+11,\, l=4\\
\end{cases}$
\end{thm}
\begin{proof}
Let the vertex set of the wheel graph $W_{5s+l}$ be defined as $V(W_{5s+l})= \{h, u_1, u_2,...,u_{5s+l}\}$, where $h$ is the hub vertex of $W_{5s+l}$ and the edge set as $E(W_{5s+l})=E(C_{5s+l})\cup \{hu_1, hu_2,...,hu_{5s+l}\}$.\\

\noindent\textbf{Case 1.} $l=0$\\
Here, any one of the following sets $\{h,v_1, v_2$, $v_4$, $v_6,v_7,v_9,...\}$, $\{h,v_2, v_3$, $v_5$, $v_7,v_8,v_{10},...\}$, $\{h,v_2, v_4,v_5,v_7,v_9,v_{10},...\}$, $\{h,v_1, v_3,v_4,v_6,v_8,v_{9},...\}$ forms the secure vertex cover for $W_{5s+l}$. If we place all the pebbles on any $v_i$ of $W_{5s+l}$, then we cannot produce a secure vertex cover solution with $12s-2$ pebbles. Thus, we conclude that $f_{svcp}(W_{5s+l})\geq 12s-1$.\\
Now, to complete the proof, it remains to prove that for any distribution of $12s-1$ pebbles on $W_{5s+l}$, we should be able to place at least one pebble on all the vertices of any one of the secure vertex cover set. Let the proof be divided into the following sub-cases based on the distribution of pebbles to the hub vertex. \\
\noindent\textbf{Case 1.1.} Hub vertex with at least $6s+1$ pebbles\\
A minimum of $6s$ pebbles are required to place at least one  pebble on all the $3s$ vertices of the secure vertex cover for the cycle $C_{5s+l}$ of $W_{5s+l}$. Eventually, by placing the one remaining pebble on the hub vertex, we are done.\\
\noindent\textbf{Case 1.2.} Hub vertex has exactly $p< 6s+1$ pebbles\\
Consequently, we are left with  at least $12s-2-(6s+1)=6s - 1$ pebbles on the vertices of the cycle $C_{5s+l}$ of $W_{5s+l}$. Among the four possible secure vertex cover sets, let us choose a set, say $S$, having more vertices with pebbles when compared to the other three possible sets. If all the four secure vertex cover sets have the same number of vertices with pebbles, then we choose a set that has more vertices with an odd number of pebbles. If not, choose any one of the sets from the possible four sets. Let $t$ be the number of vertices in the above-chosen set that have pebbles on it. In order to produce a secure cover solution, we need to place a pebble on the remaining $3s-t$ vertices of $S$ and the hub vertex. Let the number of vertices in the remaining three sets, say, $S_1$, $S_2$ and $S_3$ that have pebbles on it be $t_1$, $t_2$ and $t_3$ respectively. We observe that some vertices $v_i, 1\leq i\leq 5s$ belong to more than one set. In this case, we need to consider those vertices belongs only to the set $S$. Clearly, $t_1, t_2, t_3\leq t$ and hence $\lfloor\frac{-(t_1+t_2+t_3)}{3}\rfloor < t$ since each $t_i\geq 0$, $1\leq i\leq 3 $.\\
If possible, let us place a pebble on the $3s-t$ vertices of $S$ by a pebbling move when any of the vertices of $S$ is adjacent to a vertex of $S_1$ or $S_2$ or $S_3$ that have at least 2 pebbles on it. Let us assume that the $k$ number of vertices of set $S$ get pebbles in this way. Eventually, we use $2k$ number of pebbles in order to place a pebble on $k$ vertices of set $S$. Retain a maximum of zero or one pebble on the sets  $S_1$, $S_2$, $S_3$ and keep a maximum of one or two pebbles on $S$. Transfer all the remaining pebbles on the sets $S$, $S_1$, $S_2$ and $S_3$ to the hub vertex by a sequence of pebbling moves.\\
After transferring all the pebbles to the hub vertex we are left with a minimum of $p+\lfloor\frac{12s-1-p-2t-2k-t_1-t_2-t_3}{2}\rfloor$ =$ p+6s-t-k+\lfloor\frac{-1-p-t_1-t_2-t_3}{2}\rfloor$ pebbles on the hub vertex.\\
However, to place at least one pebble on each $3s-t-k$ vertices of set $S$ and the hub vertex, we need $1+2(3s-t-k)$ pebbles on the hub vertex.\\
Since, $ p+6s-t-k+\lfloor\frac{-1-p-t_1-t_2-t_3}{2}\rfloor-(1+2(3s-t-k))=p+t+k+\lfloor\frac{-1-p-t_1-t_2-t_3}{2}\rfloor-1>0$ as $k\leq t_1, t_2, t_3\leq t$ and $\lfloor\frac{-(t_1+t_2+t_3)}{3}\rfloor < t$, we are done.\\

\noindent \textbf{Case 2.} $l=1$\\
$\{h,v_1, v_2,v_4,v_6,v_7,v_9,...,v_{5s}\}$ or $\{h,v_1, v_2,v_4,v_6,v_7,v_9,...,v_{5s+1}\}$ or $\{h,v_2, v_3$, $v_5,v_7,v_8,v_{10},..., v_{5s+1}\}$ or $\{h,v_2, v_3,v_5,v_7,v_8,v_{10},..., v_{1}\}$ or $\{h,v_2, v_4,v_5,v_7,v_9$, $v_{10},...v_{5s+1}\}$ or $\{h,v_2, v_4,v_5,v_7,v_9,v_{10},...v_{1}\}$ or $\{h,v_1, v_3,v_4,v_6,v_8,v_{9},...v_{5s+1}\}$ or $\{h,v_1, v_3,v_4,v_6,v_8,v_{9},...v_{5s}\}$ are the different possible secure vertex cover for $W_{5s+l}$. We can't produce a secure vertex cover solution if we place all the $12s+2$ pebbles on any $v_i$ of $W_{5s+l}$. Therefore, $f_{svcp}(W_{5s+l})\geq 12s+3$.\\
\noindent \textbf{Case 2.1} Hub vertex has at least $6s+3$ pebbles\\
Here, $6s+2$ pebbles are required to place at least one pebble on all the $3s+1$ vertices of the secure vertex cover for the cycle $C_{5s+l}$ of $W_{5s+l}$. Then, by placing the left-out pebble on the hub vertex we are done.\\
\noindent \textbf{Case 2.2} Hub vertex has exactly $p < 6s+3$ pebbles. \\
Thus, the cycle $C_{5s+1}$ of $W_{5s+1}$ has at least $12s+3-(6s+3)=6s$ pebbles. Then, by proceeding as in Case 1.2 we can produce a secure vertex cover solution for $W_{5s+l}$.\\

\noindent \textbf{Case 3.} $l=2,3$\\
Consider the case of placing $12s+6$ pebbles on any $v_i$ of $W_{5s+l}, l=2,3$. Then, it is not possible to place a pebble on all the vertices of any secure vertex cover by a sequence of pebbling moves. Therefore, $f_{svcp}(W_{5s+l})\geq 12s+7, l=2,3$.\\
\textbf{Case 3.1} Hub vertex has at least $6s+5$ pebbles\\
Now, in order to place at least one pebble on $3s+2$ vertices of $C_{5s+l}$ and the hub vertex of the secure vertex cover of $W_{5s+l},l=2,3$, we need at least $6s+5$ pebbles and hence we are done. \\
\textbf{Case 3.2} Hub vertex has exactly $p< 6s+5$ pebbles\\
In this case, the cycle $C_{5s+l}$ of $W_{5s+l}, l=2,3$ has at least $12s+7-(6s+5)=6s+2$ pebbles and by proceeding as in Case 1.2 we can produce a secure vertex cover solution for $W_{5s+l},l=2,3$.\\

\noindent \textbf{Case 4.} $l=4$\\
Consider the distribution of placing $12s+6$ the pebbles on any $v_i$ of $W_{5s+l}$. Then it is not possible to produce a secure vertex cover solution and hence $f_{svcp}(W_{5s+l})\geq 12s+11$.\\
\textbf{Case 4.1} Hub vertex has at least $ 6s+7$ pebbles\\
 To place a pebble on $3(s+1)$ vertices of $C_{5s+l}$ and the hub vertex of the secure vertex cover of $W_{5s+l}$, we need at least $6(s+1)+1$ pebbles on the hub vertex and we are done. \\
\textbf{Case 4.2} Hub vertex has exactly $p<6s+7$ pebbles\\
Eventually, there exists at least $12s+11-(6(s+1)+1)=6s+4$ pebbles on the cycle $C_{5s+l}$ of $W_{5s+l}$ and by proceeding as in Case 1.2 we can produce a secure vertex cover solution for $W_{5s+4}$.\\
So, from all the above cases the result follows.
\end{proof}
\section{Conclusion}
Graph pebbling and vertex cover are the two rapidly developing areas of research in graph theory. By combining graph pebbling and secure vertex cover, one can potentially model interesting real-world problems. So, in this paper, we combined the two graph invariants cover pebbling and secure vertex cover and obtained a new graph invariant called `secure vertex cover pebbling'. We proved the NP- completeness for  `secure vertex problem' and 'secure vertex cover problem. We also discussed some basic results which relate secure vertex cover pebbling and other graph invariants. The secure vertex cover pebbling number for complete graph $K_n$, complete $r$- partite graph $K_{p_1,p_2,...p_r}$, path $P_n$, friendship graph $F_n$ and wheel graph $W_n$ are also determined.\\
Given below are some open problems.
\begin{itemize}
\item The generalized problem of placing $m$ pebbles on the secure vertex cover set.\\
Let $\pi(G,D)$ be the minimum $m$ such that every distribution of $m$ pebbles on $G$ can reach $D$. For  a vertex $v$ and natural number $t$, let $v^{t}$ denote the target distribution with $t$ pebbles on $v$ and no pebbles anywhere else, and $\pi_t(G)= max_v\pi(G,v^{t})$. Suppose that $\alpha_{s}(G)=t$, then is it true that $f_{svcp}(G)\leq \pi_{t}(G)-t+1?$ It is conjectured in \cite{AlHu} that  $\pi(G,D) \le \pi_t(G)-t+1$ for all $G$ and $D$ with $supp(D) = t$.
\item Finding the secure vertex cover pebbling number for networks as well as for directed graphs.
\item Finding the class of families of graphs $G$ with $n$ vertices such that $f_{svcp}(G)=n$. 
\end{itemize}

\end{document}